\documentclass[a4paper,twoside,12pt]{article}

\usepackage{amssymb,amsfonts,amsmath,amsthm,latexsym}
\usepackage{hyperref}
\usepackage[hmargin=1.25in,vmargin=1.25in]{geometry}
\usepackage{graphicx,caption,subcaption}
\usepackage[auth-sc-lg,affil-sl]{authblk}
\usepackage{float}
\usepackage{setspace}
\numberwithin{equation}{section}

\newtheorem{thm}{Theorem}[section]
\newtheorem{defn}{Definition}[section]
\newtheorem{lem}[thm]{Lemma}
\newtheorem{prop}[thm]{Proposition}
\newtheorem{cor}[thm]{Corollary}
\newtheorem{ill}{Illustration}

\newtheorem{conj}{Conjecture}

\def\ni{\noindent}
\def\N{\mathbb{N}}
\def\C{\mathbb{C}}

\title{\textbf{\sc On New Thue Colouring Concepts of Certain Graphs}}

\author{Johan Kok}
\affil{\small Tshwane Metropolitan Police Department\\ City of Tshwane, Republic of South Africa \\ {\tt kokkiek2@tshwane.gov.za}}

\author{Erika \v{S}krabul'\'{a}kov\'{a}}
\affil{\small Division of Applied Mathematics \\ Technical University of Ko\v{s}ice\\ Ko\v{s}ice, Slovakia\\ {\tt erika.skrabulakova@tuke.sk}}

\author{Naduvath Sudev}
\affil{\small Department of Mathematics\\ Vidya Academy of Science \& Technology \\ Thalakkottukara, Thrissur - 680501, India.\\ {\tt sudevnk@gmail.com}}

\date{}

\begin{document}
\maketitle

\begin{abstract}
Let $G$ be a simple undirected graph of order $n\ge 1$ and let $\varphi$ be a proper vertex colouring of $G$.  The Thue colouring of a graph is a colouring such that the sequence of vertex colours of any path of even and finite length in $G$ is non-repetitive.  The change in the Thue number, $\pi(G)$, as edges are iteratively removed from a graph $G$ is studied. Clearly $\pi(G-e) \le \pi(G),~ e \in E(G)$. As the edges are removed iteratively the resultant values of $\pi$ are used to construct a Thue sequence for the graph $G$. The notion of the  $\tau$-index denoted, $\tau(G)$, of a graph $G$ is introduced as well. $\tau(G)$ serves as  a measure for the efficiency of edge deletion to reduce the Thue chromatic number of a graph $G$.
\end{abstract}

\ni \textbf{Keywords:} Thue chromatic number, Thue number, Thue sequence, $\tau$-index.
\vspace{0.25cm}

\ni \textbf{AMS Classification Numbers:} 05C15, 05C55, 05D40.

\section{Introduction}

For general notation and concepts in graph and digraph theory, we refer to \cite{BM1,CL1,FH,JG,DBW}. Unless mentioned otherwise, all graphs we consider in this paper are simple, connected, finite and undirected graphs.

Recall that two vertices $u$ and $v$ are said to be \textit{adjacent} if there exist an edge between them and they are said to be \textit{reachable} from one to the other if there exist a finite path between them. In other words, two vertices $u$ and $v$ is said to be \textit{reachable} in $G$ if and only if the distance $d_G(u,v)$ is finite, where $d_G(u, v)$ is the length of shortest path between the vertices $u$ and $v$ in $G$. Since a vertex $v$ is inherently adjacent to itself,  the distance $d_G(u,u) = 0$. Consider the graph $G = G_1 \cup G_2$, the (disjoint) union of $G_1$ and $G_2$. Then, the distance between a vertex $u\in G_1$ and a vertex $v\in G_2$ is defined to be $d_G(u,v) = \infty$. 

A finite sequence $S = (q_1,q_2,q_3,\ldots ,q_t)=(q_i)_{i=1}^t$ of symbols of any alphabet is said to be \textit{non-repetitive} if for all subsequences $(r_1,r_2,r_3,\ldots,r_{2m}); m\ge 1, 2m \le t$ of symbols, the condition $r_i \ne r_{i+i}$ holds for all $1 \le i \le m$. 

The theory on graph colouring has become a fertile research area since the second half of nineteenth century. Several types of graph colouring patterns were introduced and studied over the years. In this paper, we study a particular type of vertex colouring of graphs, namely Thue Coloring. This type of graph colouring is named after the Norwegian mathematician Axel Thue, who introduced the notion of Thue colouring as follows.

\begin{defn}{\rm 
\cite{AT1} Let $G$ be a simple undirected graph of order $n\ge 1$ and let $\varphi$ be a proper vertex colouring of $G$. If the sequence of vertex colours of any path of even and finite length in $G$ is non-repetitive, then $\varphi$ is called a \textit{Thue colouring} of $G$. }
\end{defn} 

\begin{defn}{\rm 
\cite{AT1} The \textit{Thue chromatic number} of $G$, denoted by $\pi(G)$, is the minimum number of colours required for a Thue colouring of $G$.}
\end{defn} 


Certain studies on the Thue colouring of graphs have been done in \cite{AGHR,BC1,BW1,JDC1,JG1,ES2,SS1}. Motivated by these studies, we introduce new Thue colouring concepts and study these for of certain graphs.
 
\section{Thue Sequence of a Graph}

In this section, we first discuss the change in the Thue number, $\pi(G)$, as edges are iteratively removed from a graph $G$.  By definition, we have $\pi(K_1)=1$ and if $G =\bigcup_{i=1}^\ell G_i$, where each $G_i$ is a simple connected graph,  then $\pi(G)=\max\{\pi(G_i): 1 \le i\le \ell\}$. Clearly, $\pi(G-e) \le \pi(G),~ e \in E(G)$. Label the edges of $G,~ e_1, e_2, e_3,\ldots,e_{\epsilon}$, where $\epsilon$ is the number of edges in $G$. Then, we can define a Thue sequence as follows.

\begin{defn}{\rm 
Let $G=G^{\ast}_0$ and $G^{\ast}_i = G^*_{i-1}-e_j, e_j \in E(G_{i-1})$, then the sequence $(\pi(G^*_i))_{1=1}^{\epsilon}$ is called a \textit{Thue sequence} of $G$.}
\end{defn}

The following proposition provides a relation between Thue numbers of a graph $G$ and that of a graph obtained by joining some of their non-adjacent vertices.

\begin{prop}\label{Prop-2.1}
For a graph $G$, which is not complete, of order $n \ge 3$ we have $\pi(G) \le \pi(G+uv); u,v \in V(G)$.
\end{prop}
\begin{proof}
Recall that $\pi(P_3)=2$ and $\pi(C_3)=3$ (see \cite{ES2,AT1}). Hence, the result holds for non-complete graphs on $3$ vertices. Now, assume that the results holds for all non-complete graphs $G$ on $4 \le n \le k$ vertices. Consider the non-complete graph $H = G \bigcup K_1$ on $k+1$ vertices. Now, by the induction hypothesis, the result $\pi(G) = \pi(G\cup K_1) \le \pi((G + uv)\cup K_1)$ holds. Hence, we need only to check the cases by adding the edges $uv_i$, where $u$ is the vertex of $K_1$ and $v_i \in V(G)$. Add the edge $uv_i$. Clearly, by re-colouring, we have either  $\pi(G+uv_i)=\pi(G)$ or $\pi(G+uv_i)=\pi(G)+1$. By immediate induction, the  second equation holds for any number of edges $uv_i, 1\le i\le k$, added to $G$ to obtain a new graph $G'$. If $G'$ is not complete, then again by induction, it follows that $\pi(G')\le \pi(G'+uv_j)$ for all possible $v_j\in V(G')$. Hence, the result holds for all graphs, that are not complete, on $k+1$ vertices. Then, the result follows by induction.
\end{proof} 

\begin{prop}
For graph $G$ of order $n \ge 1$ we have $\pi(G-e)\le \pi(G)$ for all $e \in E(G)$. 
\end{prop}
\begin{proof}
The result follows by similar reasoning used in the proof of Proposition \ref{Prop-2.1}.
\end{proof}

Note that if a graph $G$ is a Hamiltonian graph, then $\pi(G)\ge 3$. Label the edges of $G$ in  injective manner by $e_1, e_2, e_3,\ldots ,e_{\epsilon}$. Also, let $G=G^{\ast}_0$ and $G^{\ast}_i = G^{\ast}_{i-1}-e_j,~ e_j \in E(G_{i-1})$, then the sequence $(\pi(G^{\ast}_i)_{1=1}^{\epsilon}$ is called a \textit{Thue sequence} of the edge deletion sequence. There are $\epsilon!$ such edge deletion sequences for $G$. Let $s_{\pi ,j}$, for all  $1\le j \le \epsilon!$ denote the edge deletion sequences.

\begin{defn}\label{Defn-T-IND}{\rm 
For an edge deletion sequence $(s_{\pi,j}), j\in \{1,2,3 \ldots ,\epsilon!\}$, define the \textit{$\tau$-index} by $\tau(s_{\pi,j})=\frac{1}{\epsilon+1}\sum\limits_{i=0}^{\epsilon}\pi(G^{\ast}_i)$.
}\end{defn}

Invoking the above definition, we can say that  the \textit{$\tau$-index} of a graph $G$ is the efficiency of edge deletion to reduce the Thue chromatic number of $G$. Hence, the $\tau$-index of a graph $G$ is defined as follows.

\begin{defn}\label{Defn-T-IND2}{\rm 
For a graph $G$ the $\tau$-index of $G,~ \tau(G) = \min\{\tau(s_{\pi ,j}):\forall s_{\pi ,j},1\le j \le \epsilon\,!\}$. 
}\end{defn}

\begin{ill}{\rm 
Consider the cycle $C_9$ and label the vertices by $v_1, v_2, v_3,\ldots ,v_9$ in a clockwise manner and the edges by $e_i=v_iv_{i+1}$ in the sense that $v_{n+1}=v_1$. Consider the edge deletion sequence $s_{\pi,1} = (e_1,e_2,e_3,\ldots,e_9)$. Deleting the edges sequentially results in the Thue sequence $(4,3,3,3,3,3,3,2,2,1)$. Hence, $\tau(s_{\pi,1})=\frac{27}{10}$.

Now, consider the edge deletion sequence $s_{\pi,2} = (e_3,e_6,e_2,e_4, e_6,e_8,e_1,e_5,e_7,e_9)$. The corresponding Thue sequence is given by $(4,3,3,2,2,2,2,2,2,1)$. Hence, $\tau(s_{\pi,2}) = \frac{23}{10}$. Over all $s_{\pi ,j},~ 1\le j \le 9!$ we find $\tau(C_9)=\frac{23}{10}$.}
\end{ill}

\begin{cor}
For a graph $G,~ \pi(G)=1$ if and only if $E(G)=\emptyset$. That is, the Thue chromatic number of a graph $G,$ not necessary connected, is $1$ if and only if it is an edgeless graph.
\end{cor}
\begin{proof}
Let $E(G) = \emptyset$. The result is trivially true for $K_1$.  For an edgeless graph on $n\ge 2$ vertices the colouring $\varphi: v_i \mapsto c_1, \forall~ v_i \in V(G)$ is a proper minimum vertex colouring of $G$. Also, since no distinct vertices of $G$ are reachable, $\pi(G) = \max\{\underbrace{\pi(K_1),\pi(K_1),\ldots \pi(K_1)}_{n-entries}\}$ and hence the result holds.

Now, let $\pi(G) =1$ and assume that $G$ has at least one edge say, $uv$. Since a Thue colouring is a proper vertex colouring $\varphi$ of $G$, the chromatic number $\chi(G) \ge 2$ and hence $\pi(G) \ge 2$. This is a contradiction to the hypothesis and therefore $G$ is edgeless.
\end{proof}

Since $0! = 1$ the edgeless graph has the default edge deletion sequence $s_{\pi,\emptyset}$.  For a trivial graph $G$, it follows from Definitions \ref{Defn-T-IND} and \ref{Defn-T-IND2} that $\tau(s_{\pi,\emptyset})=\tau(G)=1$.

\begin{thm}
For all graphs $G$ with $\epsilon\ge 1$, we have $\frac{3}{2}\le \min\{\tau(G)\}<2$.
\end{thm}
\begin{proof}
First, note that $\tau (P_2) = \frac{3}{2}$ and hence we have $\min\{\tau(G)\}\ge \frac{3}{2}$.
Also,  note that the unique Thue sequence of the star $K_{1,n}$ is $(\underbrace{2, 2, 2,\ldots , 2}_{n-entries},1)$. Since $\pi(G^{\ast}_i) \in \N,~ 1\le i\le \epsilon(G)$, no other Thue sequence of any other graph $G$ with $n$ edges can improve on the minimality of $\tau(K_{1,n})$. Since $\lim_{n\to \infty}\frac{2n+1}{n+1} =2$, the result $\frac{3}{2}\le \min\{\tau(G)\}< 2$ follows.
\end{proof}

Observe that for all graphs $G$ with $\epsilon(G) \ge 2$ the inequality $\tau(S_{1,n})<\tau(G)$ holds.

\subsection{$\tau$-Index of Paths and Cycles}

From \cite{ES2} it is known that $\pi(P_1) = 1,~ \pi(P_2) = \pi(P_3) = 2$ and $\pi(P_n) =3$ for $n\ge 4$. Hence, the most efficient way to reduce the Thue number of paths through edge deletion is to delete edges which create $P_1, P_2, P_3$ components. For paths $P_n, ~ n \ge 4$, the following result holds.

\begin{prop}\label{Prop-2.5}
For a path $P_n,~n \ge 4$, the $\tau$-index is given by 
\begin{equation*} 
\tau(P_n) =
\begin{cases}
\frac{7n-6}{3n}; & n=3t,~ t\in\N,\\
\frac{7n-5}{3n}; & n = 3t-1,~t \in \N,\\
 \frac{7n-4}{3n}; & n = 3t-2,~ t \in \N.
\end{cases}
\end{equation*}
\end{prop}
\begin{proof}
Consider a path $P_n,~ n \ge 4$. Label the vertices and the edges of the path from left to right, $v_1,v_2,v_3,\ldots ,v_n$ and $e_1,e_2,e_3,\ldots , e_{n-1}$, respectively.

Let $n=3t, t\in \N$. First, consider the partial edge deletion sequence $s_{\pi,1} = (e_3,e_6,e_9,\ldots , e_{3(t-2)})$. Deleting these edges sequentially results in the partial Thue sequence $(\underbrace{3,3,3,\ldots ,3}_{(t-1)-entries})$. Hence, the partial sum is equal to $3(t-1) =\frac{3n}{3}-3$. 

Now, consider the partial edge deletion sequence $s_{\pi,2} = (e_1,e_4, e_7,\ldots , e_{3(t-1)})$. Deleting these edges sequentially results in the partial Thue sequence $(\underbrace{2,2,2,\ldots ,2}_{2t-entries})$. The corresponding partial sum is equal to $4t =\frac{4n}{3}$. So far, we have deleted $(t-2) + 2t = 3t -2$ edges. Deleting the last edge adds the \textit{partial sum}, by $1$. From Definition \ref{Defn-T-IND}, it follows that $\tau(P_n)=\frac{1}{n}(\frac{3n}{3}-3+\frac{4n}{3}+1)=\frac{7n-6}{3}$.

\ni The other cases will also follow by a similar counting technique.
\end{proof}

The result of Proposition \ref{Prop-2.5}  can be re-written as $\tau(P_n)=\frac{1}{n}\left( \frac{7(n+i)}{3} -2(i+1)\right),~ i \in \{0,1,2\}$ and $n = 3t - i,~ t \in \N$. Let $\C=\{5,7,9,10,14,17\}$. From \cite{JDC1}, we have $\pi(C_n)=4,~ n\in\C$, else $\pi(C_n)=3$. 

\begin{prop}\label{Prop-2.6}
Consider a cycle $C_n, n\ge 4$, then 
\begin{equation*} 
\tau(C_n) =
\begin{cases}
\frac{7n-i+6}{3(n+1)}; & i \in \{0,1,2\}~~ \text {and}~~ n=3t-i,~t\in\mathbb{N}, n\in\C,\\
\frac{7n-i+3}{3(n+1)}; & \text {otherwise}.
\end{cases}
\end{equation*}
\end{prop}
\begin{proof}
Observe that $C_n-e,~ e \in V(C_n)$ is the path $P_n$, and hence the result follows immediately.
\end{proof}

Observe that for the corresponding line graphs, $\tau(L(P_n))=\tau(P_{n-1})$ and $\tau(L(C_n))= \tau(C_n)$. In respect of non-repetitive edge colouring, the \textit{Thue chromatic index} or \textit{Thue index} for brevity, denoted $\pi'(G)$ is defined in (see \cite{AGHR}) as a minimal, proper, non-repetitive edge colouring. Since edge colouring of the edges of a graph $G$  is equivalent to vertex colouring of the vertices of the line graph $L(G)$, all known results for $\pi(G), \pi'(G)$ and $\tau(G)$ are easily transversable between graphs and line graphs.

\subsection{$\tau$-Index of Complete Graphs}

It is known that $\pi(K_{n\ge 1}) = n$, (see \cite{ES2}). We now study the $\tau$-index of complete graphs $K_n,~ n \ge 3$. The basis for the proof of the next result through induction, can be illustrated as follows. Consider the complete graph $K_3$ and label the edges $e_1,e_2,e_3$. Deleting these edges sequentially results in the Thue sequence $(3,2,2,1)$. Extend $K_3$ to $K_4$ and remove the same edges sequentially. The partial Thue sequence is now given by $(4,3,3,2)$ and the resultant graph is $K_{1,3}$. Hence, on deleting $2$ additional edges the partial Thue sequence becomes $(4,3,3,2,2,2)$. On deleting the last edge, we get the complete Thue sequence $(4,3,3,2,2,2,1)$.

\begin{prop}
For $n\ge 2$, the $\tau$-index of a complete graph $K_n$ is
$\tau(K_n)=\frac{2}{n(n-1)}(\sum\limits_{i=0}^{n-2}(i+1)(n- i) +1)$. 
\end{prop}
\begin{proof}
We proceed by mathematical induction. The result trivially holds for $K_2$ and $K_3$. Now, assume that it holds for $K_t, ~ 2\le t \le m$. Hence, the Thue sequence of $K_m$ is given by $(m,m-1,m-1,m-2,m-2,m-2,\ldots ,\underbrace{(m-\ell),(m-\ell)\ldots ,(m-\ell)}_{(\ell+1)-entries},\\ \ldots,\underbrace{2,2,2,\ldots ,2}_{(m-1)-entries},1)$. 

Now, consider the complete graph $K_{m+1}$ and delete the edges of the subgraph $K_m$ as before and then delete the edges of the resultant graph $K_{1,m}$. Then, the Thue sequence of $K_{m+1}$ is given by $(m+1,m,m,(m-1),(m-1),(m-1),\ldots ,\underbrace{(m-\ell +1),(m-\ell +1)\ldots ,(m-\ell +1)}_{(\ell+1)-entries},\ldots,\underbrace{3,3,3,\ldots ,3}_{(m-1)-entries},2,\underbrace{2,2,2\ldots ,2}_{(m-1)-entries},1)$. This sequence can be written as $(q_i:(i+1)(n-i), 0\le i \le n-2)\cup(1)$. As $\epsilon(K_{m+1})=\frac{m(m+1)}{2}$, we have
$\tau(K_{m+1})=\frac{2}{m(m+1)}\sum\limits_{i=0}^{m-1}(i+1)((m+1)-i)+1)$. 

\ni Hence, the result follows by induction.
\end{proof}

\section{Thue Colouring of Linear Jaco Graphs}


As a generalisation of certain types of Jaco graphs, the notion of a linear Jaco graph has been introduced in \cite{KSS1} as follows.

\begin{defn}{\rm 
\cite{KSS1} Let $f(x) = mx + c; x \in \N,~ m,c\in \N_0$. The family of infinite linear Jaco graphs, denoted by $\{J_\infty(f(x)):f(x) = mx + c; x \in \N$ and $m,c \in \N_0\}$, is defined by $V(J_\infty(f(x))) = \{v_i: i \in \N\},~ A(J_\infty(f(x))) \subseteq \{(v_i, v_j): i, j \in \N, i< j\}$ and $(v_i,v_ j) \in A(J_\infty(f(x)))$ if and only if $(f(i) + i) - d^-(v_i) \ge j$.
}\end{defn}

\begin{defn}{\rm 
\cite{KSS1} The family of finite linear Jaco graphs denoted by $\{J_n(f(x)):f(x) = mx + c; x \in \N$ and $m,c\in \N_0\}$ is defined by $V(J_n(f(x))) = \{v_i: i \in \N, i \le n \},~ A(J_n(f(x))) \subseteq \{(v_i, v_j): i,j \in \N, i< j \le n\}$ and $(v_i,v_ j) \in A(J_n(f(x)))$ if and only if $(f(i)+i)-d^-(v_i)\ge j$.
}\end{defn} 


A vertex which attains the degree $\Delta(J_n(f(x)))$ is called a {\em Jaconian vertex} and the set of Jaconian vertices of the Jaco graph $J_n(f(x))$ is called the {\em Jaconian set} of  $J_n(f(x))$ and is denoted by $\mathbb{J}(J_n(x))$ or simply by $\mathbb{J}_n(x)$ for brevity, if the context is clear. The lowest numbered (subscripted) Jaconian vertex is called the \textit{prime Jaconian vertex} of a Jaco Graph.

If $v_i$ is the prime Jaconian vertex of a Jaco Graph $J_n (1)$, the complete subgraph on
vertices $v_{i+1},v_{i+2},\ldots, v_ n$ is called the \textit{Hope subgraph} or \textit{Hope graph} of a Jaco Graph and is denoted by $\mathbb{H}(J_n(x))$ or $\mathbb{H}_n(x)$ for brevity.

We refer to the corresponding underlying graph of $J_n(f(x))$, denoted $J^{\ast}_n(f(x))$, as a Jaco graph as well. We recall a result from \cite{KSS1}.

\begin{lem}\label{Lem-3.1}
{\rm \cite{KSS1}} For $m=0$ and $c\ge 0$, two special classes of disconnected linear Jaco graphs exist. For $c=0$ the Jaco graph $J_n(0)$  is a null graph (\textit{edgeless graph}) on $n$ vertices. For $c> 0$, the Jaco graph $J_n(c)=\bigcup\limits_{\lfloor\frac{n}{c+1}\rfloor-copies}K_{c+1} \bigcup K_{n-(c+1)\cdot \lfloor\frac{n}{c+1}\rfloor}$.
\end{lem}

\begin{prop}\label{Prop-3.2}
Let $m = 0$ and $c\ge 0$, then $\pi(J^{\ast}_{n\ge 1}(c)) \in \{1,2,3,\ldots,c+1\}$.
\end{prop}
\begin{proof}
From Lemma \ref{Lem-3.1}, it follows that if $c=0$ then $J^{\ast}_{n\ge 1}(0)$ is an edgeless graph. Hence, $\pi(J^{\ast}_{n\ge 1}(0))=1$. If $c> 0$ then $J^{\ast}_{n\ge 1}(c)$ is the union of complete graphs of which a largest component is $K_{c+1}$ if and only if $n \ge c+1$. Hence, $\pi(J^{\ast}_{n\ge c+1}(c))=c+1$. For $t\le c$ we have a complete graph $K_t$ thus, $\pi(J^{\ast}_{1\le t\le c}(c))=t$. This settles the result that $\pi(J^{\ast}_{n\ge 1}(c)) \in \{1,2,3,\ldots,c+1\}$.
\end{proof}
 
We now present a corresponding result for the case $m=1,c=0$. First, recall a lemma from \cite{KSS1}.

\begin{lem}\label{Lem-3.3}
{\rm \cite{KSS1}} For a Jaco graph $J^{\ast}_{n+1}(x)$ we have, $\Delta(J^{\ast}_{n+1}(x)) = \Delta(J^{\ast}_n(x))$ or $\Delta(J^{\ast}_n(x)) + 1$.
\end{lem}

\begin{thm}\label{Thm-3.4}[Bokang's theorem\footnote{The first author dedicates this theorem to Miss Bokang Lerato Tshabalala who is expected to grow up with a deep fond of mathematics.}]
For the Jaco graph $J^{\ast}_n(x), ~n\ge 1$ we have $\Delta(J^{\ast}_n(x)) \le \pi(J^{\ast}_n(x)) \le \Delta(J^{\ast}_n(x)) + 2$.
\end{thm}
\begin{proof}
For $n = 1,2,3,4$ the corresponding Jaco graphs are paths so the result holds. Since the Hope graph of $J^{\ast}_n(x),~ \forall n\in \N$ is a complete graph $K_{\Delta(J^{\ast}_n(x))}$, the lower bound is obvious. For $J^{\ast}_5(x)$, we have the path $v_1,v_2,v_3,v_4,v_5$ with the added edge $v_3v_5$. Clearly, $\pi(J^{\ast}_5(x))=4=\Delta(J^{\ast}_5(x))+1$. Assume that it holds for all Jaco graphs $J^{\ast}_n(x), ~5 \le n \le k$ and let the vertices of $J^{\ast}_k(x)$ allow the Thue colouring $\varphi:V(J^{\ast}_k(x)) \mapsto \{c_1,c_2,c_3, \ldots, c_\ell\},~ \ell \le \Delta(J^{\ast}_k(x)) + 1$ with $\varphi:v_1 \mapsto c_1$. Now, consider $J^{\ast}_{k+1}(x)$.  Since $\Delta(J^{\ast}_{k+1}(x)) = \Delta(J^{\ast}_k(x))$ or $\Delta(J^{\ast}_k(x)) + 1$ and $\pi(J^{\ast}_k(x)) = \Delta(J^{\ast}_k(x))$ or $\Delta(J^{\ast}_k(x)) + 1$, we consider the following four cases. Let $\varphi = \varphi'$ for $v_i, 1 \le i \le k$ and  $\varphi':v_{k+1} \mapsto c_{\ell + 1}$.

\textit{Case 1:} Let $\ell = \pi(J^{\ast}_k(x)) = \Delta(J^{\ast}_K(x))$; $\Delta(J^{\ast}_{k+1}(x)) = \Delta(J^{\ast}_k(x))$. Since $\varphi':v_{k+1} \mapsto c_{\ell+1}$ might not allow a minimum Thue colouring, $\pi(J^{\ast}_{k+1}(x)) \le \Delta(J^{\ast}_{k+1}(x)) + 1$.

\textit{Case 2:}  Let $\ell = \pi(J^{\ast}_k(x)) = \Delta(J^{\ast}_K(x))$; $\Delta(J^{\ast}_{k+1}(x)) = \Delta(J^{\ast}_k(x)) + 1$. Since $\Delta(J^{\ast}_{k+1}(x)) = \Delta(J^{\ast}_k(x)) + 1$, the colouring $\varphi':v_{k+1} \mapsto c_{\ell+1}$ allows a minimum Thue colouring so $\pi(J^{\ast}_{k+1}(x)) = \Delta(J^{\ast}_{k+1}(x)) + 1$.

\textit{Case 3:}  Let $\ell =\pi(J^{\ast}_k(x)) = \Delta(J^{\ast}_k(x)) + 1$; $\Delta(J^{\ast}_{k+1}(x)) = \Delta(J^{\ast}_k(x)) + 1$.  Since $\varphi':v_{k+1} \mapsto c_{\ell+1}$ might not allow a minimum Thue colouring, $\pi(J^{\ast}_{k+1}(x)) \le \Delta(J^{\ast}_{k+1}(x)) + 1$.

\textit{Case 4:} Let $\ell =\pi(J^{\ast}_k(x)) = \Delta(J^{\ast}_k(x)) + 1$; $\Delta(J^{\ast}_{k+1}(x)) = \Delta(J^{\ast}_k(x))$.  The smallest Jaco graph for which we have this case is $J^{\ast}_5(x)$ and $J^{\ast}_6(x)$. For $\varphi':v_1 \mapsto c_1, \varphi':v_2 \mapsto c_2,\varphi':v_3 \mapsto c_1,\varphi':v_4 \mapsto c_3,\varphi':v_5 \mapsto c_4,\varphi':v_6 \mapsto c_1$, a minimum Thue colouring is allowed and $\pi(J^{\ast}_6(x)) \le \Delta(J^{\ast}_6(x)) + 1$. Extend $\varphi'$ to $J^{\ast}_{k+1}(x)$ with $\varphi':v_{k+1} \mapsto c_1$. Without loss of generality, assume that there exists a path beginning at vertex $v_{k+1}$ which represents at least a two block repetitive sequence say, $(c_1,c_j,c_k,c_m,\ldots ,c_t,c_1,c_j,c_k,c_m, \ldots ,c_t)$, then $(c_j,c_k,c_m,\ldots ,c_t,c_j,c_k,c_m, \ldots ,c_t)$ corresponds to a two block repetitive sequence in $J^{\ast}_k(x)$ which is a contradiction. Equally so, if a two block repetitive sequence $(c_j,c_k,c_m,\ldots ,c_t,c_1 ,c_j,c_k,c_m, \ldots,c_t,c_1)$, then $(c_j,c_k,c_m,\ldots,c_t, c_j,c_k,c_m,\ldots ,c_t)$ corresponds to a two block repetitive sequence in $J^{\ast}_k(x)$, which is a contradiction. Hence, by induction, it follows that no such repetitive sequence of any number of blocks exists. 

\ni Therefore, $\Delta(J^{\ast}_n(x)) \le \pi(J^{\ast}_n(x)) \le \Delta(J^{\ast}_n(x)) + 1$. 
\end{proof}

We strongly believed that a generalisation of Theorem \ref{Thm-3.4} is possible as provided in the following conjecture.

\begin{conj}\label{Conj-3.5}
For the Jaco graph $J^{\ast}_n(f(x)),~f(x) = mx +c, x \in \N,~m\ge 2,~ c\ge 1$, we have $\pi(J^{\ast}_n(f(x))) \le \Delta(J^{\ast}_n(f(x))) + \delta(J^{\ast}_n(f(x)))$.
\end{conj}

\section{On Thue Distance, Thue Diameter, Thue Radius and Thue Reach}

A Thue colouring of a graph $G$ has the property that it is a proper vertex colouring for which any path in $G$ is a non-repetitive sequence of colours. Therefore, the property of minimality \textit{per se} is not embedded in the concept of a Thue colouring. Hence, a rainbow colouring is a Thue colouring, but not a minimum Thue colouring. We call a minimum Thue colouring of $G$ a $\mathfrak{T}$-colouring of $G$.

Hence, if $G$ and $\pi(G)$ are known, and a finite number $t$ colours may be chosen from then, $\binom{t}{\pi(G)}$ sets of colours are available to choose from. Call these sets, $\mathfrak{T}$-sets. For any $\mathfrak{T}$-set different mappings $\varphi_1, \varphi_2, \ldots , \varphi_k$ exist, all of which will ensure a minimum Thue colouring of $G$. So for a chosen $\mathfrak{T}$-set it is true that up to equivalence the many other $\mathfrak{T}$-sets become irrelevant. Consider a graph $G$ of order $n$ and a corresponding $\mathfrak{T}$-colouring.

\begin{defn}{\rm 
For a graph $G$ with $\pi(G) < n$ and a colour $c_i$, the \textit{Thue colour distance}, denoted by $\pi_d(c_i)$, is defined to be $\pi_d(c_i) = \min\{d_G(u,v): u,v \in V(G), \varphi(u) = \varphi(v) = c_i\}$.
}\end{defn}

\begin{defn}{\rm 
For a graph $G$ with $\pi(G) < n$, the \textit{Thue radius} of a graph $G$, denoted by $\pi_r(G)$, is defined to be $\pi_r(G) = \min\{\pi_d(c_i): \forall \varphi,~ \forall \pi_d(c_i)>0 \}$.
}\end{defn}

\begin{defn}{\rm 
For a graph $G$ with $\pi(G) < n$ and a colour $c_i$, the \textit{Thue colour diameter}, denoted by $\pi_D(c_i)$, and is defined as $\pi_D(c_i)=\max\{d_G(u,v): u,v \in V(G),\varphi(u)=\varphi(v)= c_i\}$.
}\end{defn}

\begin{defn}{\rm 
For a graph $G$ with $\pi(G) < n$, the \textit{Thue reach} of a graph $G$, denoted by $\pi_R(G)$,  is defined to be  $\pi_R(G) = \max\{\pi_D(c_i)\}$, for all $\varphi$ and for all $\pi_D(c_i)>0 \}$.
}\end{defn}

In view of the property of inherent adjacency, we can define the default value $\pi_d(c_i)=\pi_D(c_i)= 0$ for a rainbow colouring. Obviously, $\pi_d(c_i)\ne 1$ and hence $\pi_D(c_i) \ne 1$.

\begin{ill}{\rm 
To ensure a $\mathfrak{T}$-colouring of $P_4$, let the colours to be $\mathcal{C} = \{c_1,c_2,c_3\}$. There are many possible colourings ($\varphi$'s). Six correspond to $\pi_d(c_i) > 0$ for any $i \in \{1,2,3\}$.

\begin{enumerate}\itemsep0mm
\item[(i)] $\varphi(v_1) = c_1,\varphi(v_2) = c_2, \varphi(v_3) = c_1, \varphi(v_4) = c_3$.
\item[(ii)] $\varphi(v_1) = c_1,\varphi(v_2) = c_3, \varphi(v_3) = c_1, \varphi(v_4) = c_2$.
\item[(iii)] $\varphi(v_1) = c_2,\varphi(v_2) = c_1, \varphi(v_3) = c_2, \varphi(v_4) = c_3$.
\item[(iv)] $\varphi(v_1) = c_2,\varphi(v_2) = c_3, \varphi(v_3) = c_2, \varphi(v_4) = c_1$.
\item[(v)] $\varphi(v_1) = c_3,\varphi(v_2) = c_1, \varphi(v_3) = c_3, \varphi(v_4) = c_2$.
\item[(vi)] $\varphi(v_1) = c_3,\varphi(v_2) = c_2, \varphi(v_3) = c_3, \varphi(v_4) = c_1$.
\end{enumerate}

For (i), $\pi_d(c_1) = 2,~ \pi_d(c_2) =0,~ \pi_d(c_3) = 0$. Hence $\pi_r(P_4) = 2$. Similarly, for (ii), (iii), (iv), (v), (vi), $\pi_r(G)$ is constant regardless of $\varphi$. 

Similarly, to ensure the same $\mathfrak{T}$-colouring of $P_4$ determines $\pi_R(P_4)$, consider the following.

\begin{enumerate}\itemsep0mm
\item[(i)] $\varphi(v_1) = c_1,\varphi(v_2) = c_2, \varphi(v_3) = c_3, \varphi(v_4) = c_1$.
\item[(ii)] $\varphi(v_1) = c_1,\varphi(v_2) = c_3, \varphi(v_3) = c_2, \varphi(v_4) = c_1$.
\item[(iii)] $\varphi(v_1) = c_2,\varphi(v_2) = c_1, \varphi(v_3) = c_3, \varphi(v_4) = c_2$.
\item[(iv)] $\varphi(v_1) = c_2,\varphi(v_2) = c_3, \varphi(v_3) = c_1, \varphi(v_4) = c_2$.
\item[(v)]  $\varphi(v_1) = c_3,\varphi(v_2) = c_1, \varphi(v_3) = c_2, \varphi(v_4) = c_3$.
\item[(vi)]  $\varphi(v_1) = c_3,\varphi(v_2) = c_2, \varphi(v_3) = c_1, \varphi(v_4) = c_3$.
\end{enumerate}

For (i), $\pi_D(c_1)= 3,~ \pi_D(c_2) = 0,~ \pi_D(c_3) = 0$, hence $\pi_R(P_4) = 3$. Similarly, for (ii), (iii), (iv), (v), (vi), $\pi_R(G)$ is constant regardless of $\varphi$.}	\qed
\end{ill}

These concepts can find application if different colours represent different technologies to be installed at the points of a network and there is a requirement that a set of a particular technology should be minimum distance or maximum distance apart. For example, certain communication technology is such that the efficiency deteriorates if units are to far apart and a maximum distance criteria is set. On the other hand, for certain nuclear technology a minimum distance criteria is set. The next lemma is perhaps obvious but useful for further results.

\begin{lem}\label{Lem-4.1}
A graph $G$ of order $n\ge 3$ is a non-complete graph if and only if a subset of three vertices say, $\{u,v,w\}$ exists such that the induced subgraph $\langle u,v,w\rangle \ne C_3$.
\end{lem}
\begin{proof}
Since a disconnected graph cannot be complete the result holds for disconnected graphs. Assume that $G$ is connected.

First, assume that there exist three vertices $v,u,w$ such that $\langle v,u,w\rangle \ne C_3$ then at least $vu \notin E(G)$ or $vw \notin E(G)$ or $uw \notin E(G)$ so $G$ is not complete.

Conversely, assume that $G$ is a non-complete graph. For any two vertices $v,w$ for which the edge $vw \notin E(G)$, consider the shortest path $P = vu_1u_2u_3\ldots u_kw$. If $w=u_2$, then $\langle v,u_1,u_2\rangle \ne C_3$ and the result follows. However if $u_2 \ne w$, then the edge $vu_2 \notin E(G)$ else the path $P' = vu_2u_3\ldots w$ is a shortest path which is a contradiction. Hence, the induced subgraph $\langle v, u_1,u_2\rangle \ne C_3$.
\end{proof}

In view of the above results, the existence of Thue colouring for graphs, that are not complete, is discussed in the following theorem.

\begin{thm}
For a graph $G \ne K_n,~ \pi_r(G) = 2$. In other words, for a graph $G \ne K_n$ there exists a $\mathfrak{T}$-colouring $\varphi$ such that, vertices $u,v$ exist such that $\varphi(u) = \varphi(v)$ and $d_G(u,v) = 2$.
\end{thm}
\begin{proof}
Consider a graph $G \ne K_n$ and a corresponding $\mathfrak{T}$-colouring $\varphi:V(G) \rightarrow \mathcal{C},~ \mathcal{C} =\{c_i:1 \le i \le \pi(G)\}$ of $G$. Since $G$ is connected, it has $|V(G)|\ge 3$ and at least one subgraph, $P_3$ (see Lemma \ref{Lem-4.1}). Let $P_3 = v_1v_2v_3$.  If $\varphi(v_1) = \varphi(v_3)$ the result follows. Else, assume without loss of generality that $\varphi(v_1) = c_1,~ \varphi(v_2)=c_2,~ \varphi(v_3)=c_3$. Consider the graph $G-v_2$. Hence, if need be, through colour index rotation it is possible to Thue colour (not necessary minimum) the vertices of $G-v_2$ such that $\varphi(v_1)=c_1$ and $\varphi(v_3)=c_1$ and $\varphi: v_j \mapsto c_k,~ c_k \in \{ c_i: 1 \le i \le \pi(G)\},~ j\ne 1,3$. This allows a $\mathfrak{T}$-colouring of $G$ such that $\pi_d(c_1)=\min\{d_G(u,v):u,v\in V(G),\varphi(u)=\varphi(v)=c_1\}=d_G(v_1,v_3)=2$. Therefore, $\pi_r(G)=2$.
\end{proof}

It can be noted that the above theorem can be generelised for any graph $G$ as $\pi_r(G) \in \{0,2\}$.

\subsection{On $\pi_R(G)$}
Observe that for any non-complete graph $G,~ \pi_R(G) \le diam(G)$. It then follows that $\pi_R(G + e) \le \pi_R(G)$. Except for $K_n$ for which $\pi_R(K_n) = 0 < 1 = diam(K_n)$ no other graphs have been found as yet for which inequality holds. Hence, it is strongly suspected that $\pi_R(G) = diam(G)$. A proof of the latter still escapes us. We shall determine $\pi_R(G)$ for some classes of graphs.
\begin{prop}
\begin{enumerate}
\item[(i)] For $n\ge 3$, we have $\pi_R(P_n)=diam(P_n)=n-1$.
\item[(ii)] For $n\ge 4$, we have $\pi_R(C_n)=diam(C_n)=\lfloor\frac{n}{2}\rfloor$.
\item[(iii)] For a star $S_{n+1},~ n\ge 2$, we have $\pi_R(S_{n+1})= diam(S_{n+1})=2$.
\end{enumerate}
\begin{proof}
{\em Part-(i):} Without loss of generality let $P_n = v_1v_2v_3\ldots v_n$ and let $\varphi(v_1) = c_1$ and $c_i \in \{c_1,c_2, c_3\}$. Up to equivalence the path $P_3$ has a unique $\mathfrak{T}$-colouring, $\varphi(v_1) = c_1, \varphi(v_2) = c_2, \varphi(v_3) = c_1$ hence, $\pi_R(P_3) =diam(P_3) = 2 = 3-1$. Assume that the result holds for $P_n,~  3 \le n \le k$. It implies that $\varphi(v_k) = c_1$ in $P_k$. Consider a $\mathfrak{T}$-colouring of the path $P_{k-1}$ such that $\varphi(v_{k-1}) = c_1$ which is possible by the induction assumption. It implies that $\varphi(v_{k-2}) \in \{c_2,c_3\}$. Colour $\varphi(v_{k+1}) = c_1$ and $\varphi(v_k) \ne \varphi(v_{k-2})$. Hence, a $\mathfrak{T}$-colouring of $P_{k+1}$ exists such that $\pi_R(P_{k+1}) = diam(P_{k+1}) = k$. Therefore, the result is settled through induction.

{\em Part (ii)(a):} Consider two copies of path $P_n,~ n\ge 3$ and label them $P'_n,~ P''_n$ respectively. Let $\mathcal{C}' = \{c_1,c_2,c_3 \}$ and $\mathcal{C}'' = \{c_1,c'_2,c'_3 \}$ allow a $\mathfrak{T}$-colouring of $P'_n,~ P''_n$ respectively such that the four end vertices carry the colour $c_1$. From Part (i) the latter is always possible. Merge the end vertices pairwise (one each from a different path) to obtain the cycle $C_{2(n-1)}$. Certainly the cycle has been Thue coloured. Recolour the remaining vertices from say $P''_n$ to obtain a $\mathfrak{T}$-colouring for $C_{2(n-1)}$ which is always possible. It follows that $\pi_D(c_1) = n-1$ which is a maximum over all colours. Hence, $\pi_R(C_{2(n-1)}) = diam(C_{2(n-1)}) = \frac{2(n-1)}{2} = \lfloor\frac{2(n-1)}{2}\rfloor = n-1$.

{\em Part (ii)(b):} Consider two copies of path $P_n,~ n\ge 3$ and label them $P'_n,~ P''_n$ respectively. Let $\mathcal{C}' = \{c_1,c_2,c_3 \}$ and $\mathcal{C}'' = \{c_1,c'_2,c'_3 \}$ allow a $\mathfrak{T}$-colouring of $P'_n,~ P''_n$ respectively such that three end vertices carry the colour $c_1$ and one end vertex from say $P''_n$ carry the colour $c'_2$. Merge one end vertex from $P'_n$ and one from $P''_n$ which carries the colour $c_1$. Add an edge to join the remaining two end vertices. Clearly the cycle $C_{2n-1}$ has been Thue coloured. Recolour the remaining vertices of say $P''_n$ to obtain a $\mathfrak{T}$-colouring for $C_{2n-1}$. It follows that $\pi_D(c_1) = n-1$ which is a maximum over all colours. Hence, $\pi_R(C_{2n-1}) = diam(C_{2n-1}) = \lfloor\frac{2n-1}{2}\rfloor$.\\
Finally, since any cycle $C_n,~ n \ge 4$ can be constructed from either Parts (ii)(a) or (ii)(b), the result follows.\\
Part (iii): For $n=2$ the star $S_{2+1} \simeq P_3$ so the result follows from Part (i). Consider $n \ge 3$. Since it trivially holds for $S_{3+1}$ the result follows through immediate induction.
\end{proof}
\end{prop}

\subsection{On Subdivisions of a Graph}

 For a graph $G$ of order $n$ and vertices $v_1,v_2,v_3,\ldots ,v_n$, a $k$-subdivision  is obtained by inserting $k$ additional vertices $u_\ell,~ 1\le \ell \le k$ to each edge $v_iv_j \in E(G)$. In other words each edge $v_iv_j$ is replaced by the path $v_i u_1u_1u_3\ldots u_k v_j$. Denote the subdivision of $G$ by $G_k$. 
 
The second subdivision is called the \textit{cycle subdivision}. A distinct pair of vertices $u,w \in V(C_m)$ with $d_{C_m}(u,w) = diam(C_m)$, is called \textit{diameter vertices} of the cycle $C_m$. A cycle subdivision of a graph $G$ is the graph obtained by inserting a cycle $C_m$ into each edge with the edges $v_iu$ and $v_jw$, replacing the edge $v_iv_j$ and vertices $u,w$ are diameter vertices of $C_m$. Denote a cycle subdivision of $G$ by $G_{\circ m}$.

\begin{thm}
For a graph $G$ we have
\begin{enumerate}\itemsep0mm
	\item[(i)] $\pi(G_k)\le \pi(G) + 1$, for $k \ge 3$. 
	\item[(ii)] $\pi(G_{\circ m}) \le \pi(G) + 2$, for $m \in \{5, 7, 9, 10, 14, 17\}$ else, $\pi(G_{\circ m}) \le \pi(G) + 1$.
\end{enumerate}
\end{thm}
\begin{proof}
\textit{Part (i):} Let a $\mathfrak{T}$-colouring of $G$ be allowed by $\mathcal{C} = \{c_1,c_2,c_3,\ldots ,c_{\pi(G)}\}$ and let a $\mathfrak{T}$-colouring of $P_k,~ k \ge 3$ be allowed by $\mathcal{C}' = \{c'_1,c'_2,c'_3\}$ such that $\mathcal{C}\cap \mathcal{C}' = \emptyset$. Consider the subdivision $G_k$ with the initial colouring $\mathcal{C} \cup \mathcal{C}'$. For each edge $v_iv_j \in E(G)$ recolour the corresponding vertices of the path which carry the colour $\varphi '(u_1)$ to the colour $\varphi(v_j)$. Similarly recolour the corresponding vertices of the path which carry the colour $\varphi'(u_k)$ to the colour $\varphi(v_i)$. Clearly the new colouring is a Thue colouring, perhaps not minimum. Exactly one new colour $c'_i \in \mathcal{C}'$ has been added to the initial Thue colouring of $G$ to Thue colour $G_k$. Therefore, for $k \ge 3,~ \pi(G_k)\le \pi(G) + 1$.

\textit{Part (ii):} The second part can also be proved using the similar reasoning as in Part (i).
\end{proof}

Note that in the construction of the cycle subdivision there is no necessity to link the diameter vertices. The results holds for any distinct pair of vertices $u_i, u_j \in V(C_m)$ if $d_{C_m}(u_i,u_j) \ge 2$.

\begin{cor}
Consider paths $P_t,~ t \in \{3,4,5,\ldots , k\}$ and  insert any path $P_t$ into an edge of $G$. The varied subdivision is denoted $G_{k^{\ast}}$. For $k \ge 3$ we have $\pi(G_{k^{\ast}})\le \pi(G) + 1$.
\end{cor}

\begin{cor}
Consider cycles $C_t,~ t \in \{3,4,5,\ldots , m\}$ and insert any cycle $C_t$ into an edge of $G$. The varied subdivision is denoted $G_{\circ m^{\ast}}$. For $m \in \{5, 7, 9, 10, 14, 17\}$ we have $\pi(G_{\circ m^{\ast}}) \le \pi(G) + 2$ else, $\pi(G_{\circ m^{\ast}}) \le \pi(G) + 1$.
\end{cor}

The above leads to a general result. We begin by defining the $H$-subdivision of a given graph $G$.

\begin{defn}{\rm 
Let the graph $H$ have at least one distinct pair of vertices $u,w$ such that $d_H(u,w) \ge 2$. The \textit{$H$-subdivision} of a graph $G$, denoted $G_H$, is the graph obtained by inserting the graph $H$ into each edge $v_iv_j \in E(G)$ with the edges $v_iu$ and $v_jw$ replacing the edge $v_iv_j$.
}\end{defn}

The following theorem establishes the relation between the Thue chromatic numbers of the graphs $G$ and $H$ and the $H$-subdivision $G_H$ of $G$.

\begin{thm}
Let $H$ be a graph with at least two distinct vertices $u,w,~ d_H(u,w) \ge 2$. Then, for a graph $G$ we have $\pi(G_H)\le \pi(G)+\pi(H)-1$.
\end{thm}
\begin{proof}
 Let a $\mathfrak{T}$-colouring of $G$ be allowed by $\mathcal{C} = \{c_1,c_2,c_3,\ldots ,c_{\pi(G)}\}$ and let a $\mathfrak{T}$-colouring of $H$ be allowed by $\mathcal{C}' = \{c'_1,c'_2,c'_3 \ldots ,c_{\pi(H)}\}$ such that $\mathcal{C}\cap \mathcal{C}' = \emptyset$. Consider the $H$-subdivision $G_H$ with the initial colouring $\mathcal{C} \cup \mathcal{C}'$. For each edge $v_iv_j \in E(G)$ recolour the corresponding vertices of H which carry the colour $\varphi '(u)$ to the colour $\varphi(v_j)$. Similarly recolour the corresponding vertices of $H$ which carry the colour $\varphi'(w)$ to the colour $\varphi(v_i)$. Clearly the new colouring is a Thue colouring, perhaps not minimum. Thus far, at most two new colours $c'_\ell,c'_{\ell'} \in \mathcal{C}'$ have been removed from the initial Thue colouring of $G_H$. Since this new Thue colouring might not be minimum, $\pi(G_H) \le \pi(G) + \pi(H) - 1$. 
\end{proof}

\section{Conclusion}

So far, we have discussed the Thue colouring of certain graph classes and studied certain parameters of certain graphs, with respect to this type of graph colouring. There are so many open problems in this for further investigations.

Since a closed formula for the number of edges of $J^{\ast}_n(f(x))$ is not known yet, a closed formula for $\tau(J^{\ast}_n(f(x)))$ also remains open. Using the results on determining the number of edges, provided in \cite{KSS1}, $\tau$-index can be determined recursively. The factor $\sum\limits_{i=0}^{\epsilon(J^{\ast}_n(x))}\pi(J^{\ast}_{n,_i})$ is edge dependent as well. Hence, it is desirable to seek a recursive formula for the aforesaid factor, in the absence of a closed formula for the number of edges of a Jaco graph. It is believed that a clear understanding of the Thue number and Thue sequence of caterpillar graphs will be the foundation for finding such a recursive formula. Other Thue related invariants in respect of Jaco graphs deserve intensive further investigations.

Certainly the new notions such as Thue sequence, $\tau$-index, Thue distance, Thue diameter, Thue radius and Thue reach, also offer a wide scope of further studies.

\end{document}